\documentclass[10pt]{amsart}

\usepackage[margin=1.5in]{geometry}                
\usepackage{graphicx}
\usepackage{amssymb}
\usepackage{epstopdf}
\usepackage{color}
\usepackage{xcolor}
\usepackage{ulem}
\usepackage{enumerate}
\usepackage{mathtools}
\usepackage{enumitem}
\usepackage{amsmath}
\usepackage{braket}
\usepackage{mathrsfs}  
\usepackage{hyperref}

\usepackage [english]{babel}
\usepackage [autostyle, english = american]{csquotes}
\MakeOuterQuote{"}

\newcommand{\R}{\mathbb{R}}
\newcommand{\N}{\mathbb{N}}
\newcommand{\Z}{\mathbb{Z}}

\newcommand{\Q}{\mathbb{Q}}

\newtheorem{theorem}{Theorem}[section]

\newtheorem{fact}[theorem]{Theorem}

\newtheorem{lemma}[theorem]{Lemma}

\newtheorem{obs}[theorem]{Observation}
\newtheorem{ex}[theorem]{Example}

\newtheorem{fact1}{Theorem}
\newtheorem{cor1}[fact1]{Corollary}

\title{On the Equivalence of Equilibrium and Freezing States in Dynamical Systems}

\date{16 Apr 2025}
\begin{document}

\maketitle
\begin{center}
\author{C. Evans Hedges}
\end{center}

\begin{abstract} This paper is concerned with freezing phase transitions in general dynamical systems. A freezing phase transition is one in which, for a given potential $\phi$, there exists some inverse temperature $\beta_0 > 0$ such that for all $\alpha, \beta > \beta_0$, the collection of equilibrium states for $\alpha \phi$  and $\beta \phi$ coincide. In this sense, below the temperature $1 / \beta_0$, the system "freezes" on a fixed collection of equilibrium states.

We show that for a given invariant measure $\mu$, it is no more restrictive that $\mu$ is the freezing state for some potential than it is for $\mu$ to be the equilibrium state for some potential. In fact, our main result  applies to any collection of equilibrium states with the same entropy. In the case where the entropy map $h$ is upper semi-continuous, we show any ergodic measure $\mu$ can be obtained as a freezing state for some potential. 

In this upper semi-continuous setting, we additionally show that the collection of potentials that freeze at a single state is dense in the space of all potentials. However, in the $\Z$ action setting where the dynamical system satisfies specification, the collection of potentials that do not freeze contains a dense $G_\delta$. 
\end{abstract}

\section{Introduction}

\subsection{Dynamical Systems and Equilibrium States} 

We begin extremely generally, letting $X$ be a compact metric space and $G$ be a topological (semi)-group. We let $G$ act continuously on $X$ by $T$. I.e. $\{ T_g \}_{g \in G}$ forms a collection of continuous maps from $X$ to itself such that $T_g \circ T_h = T_{gh}$. Let $M_T(X)$ represent the set of all $T$-invariant Borel probability measures on $X$. It is well known that $M_T(X)$ forms a Choquet simplex whose extreme points are exactly the ergodic measures on $X$. In this paper we assume that $M_T(X)$ is non-empty and a notion of measure theoretic entropy is well defined, as in the case where $G$ is an amenable group. We also assume that the topological entropy of the dynamical system $(X, T)$ is finite. 

For a given $\mu \in M_T(X)$, we denote the Kolmogorov-Sinai entropy of $\mu$ by $h(\mu)$. The notion of entropy has been generalized by adding an "energy" component: we say a real-valued continuous function $\phi : X \rightarrow \R$ is a potential, and for a given potential $\phi \in C(X)$, the pressure of $\phi$ can be defined by 
$$P(\phi) = \sup  \{h(\mu) + \int \phi d\mu : \mu \in M_T(X) \}. $$
The pressure of a system $(X, T)$ under the potential $\phi$ can be thought to correspond with a scalar multiple of the free energy of the system. As such, we say $\mu$ is an equilibrium state for $\phi$ if $\mu$ attains the supremum in the definition of pressure. 

Whenever the entropy map $h$ is weak-$*$ upper semi-continuous, for example when $(X, T)$ is expansive as is the case for many symbolic systems and Anosov diffeomorphisms, it is well known that there exists at least one equilibrium state for any potential $\phi$. On the other hand, for a given dynamical system $(X, T)$ and a fixed ergodic $\mu \in M_T(X)$, it is unclear whether or not $\mu$ can be obtained as an equilibrium state for some potential $\phi$. In \cite{ruelle}, Ruelle showed that when $h$ is upper semi-continuous, any ergodic measure can be obtained as an equilibrium state: 

\begin{fact}(Corollary 3.17 \cite{ruelle}) Suppose $h$ is upper semi-continuous and let $\mathcal{E} = \{ \mu_1, \dots, \mu_n \}$ be any finite collection of ergodic measures. Then there exists a potential $\phi \in C(X)$ such that $\mathcal{E}$ is exactly the collection of ergodic equilibrium states for $\phi$. 
\end{fact}

This was extended by Jenkinson, who showed in \cite{jenkinson} that for any collection of ergodic measures $\mathcal{E}$ that is weak-$*$ closed in $M_T(X)$, there exists a potential $\phi \in C(X)$ such that $\overline{co}(\mathcal{E})$ is exactly the collection of equilibrium states for $\phi$, 
where $\overline{co}(\mathcal{E})$ represents the closed convex hull of $\mathcal{E}$. 

In the case where $h$ is not upper semi-continuous, substantially less is known. 
It is easy to show that if $\mu$ is an equilibrium state, then $h$ is upper semi-continuous at $\mu$, however it is not known if the converse is true. As such, a characterization of the properties required for a given ergodic measure to be an equilibrium state is still an open question.

\subsection{Zero Temperature Limits and Freezing Phase Transitions} 

For a given potential $\phi$ we can incorporate temperature into our analysis by examining the one-parameter family of potentials $\beta \phi$ for $\beta > 0$. Here $\beta$ represents the inverse temperature of our system and so by letting $\beta \nearrow \infty$, we observe the behavior of the system as it cools to zero-temperature. 

For each $\beta > 0$, let $\mu_\beta$ be an equilibrium state for $\beta \phi$. We say $\mu_\infty$ is a zero temperature limit (ground state) for $\phi$ if $(\mu_\beta)$ weak-$*$ converges to $\mu_\infty$ as $\beta \nearrow \infty$. In \cite{chazottes-hochman}, Chazottes and Hochman constructed a Lipschitz potential on the full shift $(\{0, 1\}^\Z, \sigma)$ whose zero-temperature limit does not exist. We let $Max(\phi) = \sup_{\mu \in M_T(X)} \int \phi d\mu$ be the maximal integral of $\phi$ against $T$-invariant measures, and define $h_\infty(\phi) = \sup \{ h(\mu) : \int \phi d\mu = Max(\phi) \}$ to be the supremum of the entropy of $\phi$-maximizing measures. In the case that the zero-temperature limit does exist, it is easy to show that $\int \phi d\mu_\infty = Max(\phi)$ and $h(\mu_\infty) = h_\infty(\phi)$. In this sense, zero-temperature limits are entropy maximizing among $\phi$-maximizing measures. 

We now define a modified pressure function for any fixed $\phi \in C(X)$. For $\beta \geq 0$, let $P_\phi(\beta) = P(\beta \phi)$. It can be shown that $P_\phi$ is convex and approaches a slant asymptote $f(\beta) = \beta \cdot Max(\phi) + h_\infty(\phi)$. In \cite{kucherenko-quas}, Kucherenko and Quas showed that for mixing $\Z$-SFTs (a large class of symbolic systems that will not be defined here) and any H\"older continuous potential $\phi$ that is not cohomologous to a constant, $P_\phi(\beta)$ approaches its slant asymptote at most exponentially quickly. Additionally, they were able to construct H\"older continuous potentials witnessing asymptotic behavior approaching the slant asymptote as slowly as desired. In these systems, the equilibrium states cool to their zero temperature limit at a relatively slow rate. This paper is concerned with the opposite phenomenon. 

The notion of a freezing phase transition has been explored in the literature, for example in  \cite{bruin-quasicrystal} and \cite{bruin-freezing}. For a given dynamical system $(X, T)$ and potential $\phi$, $\phi$ freezes at $\mu$ if there exists some $\beta_0 > 0$ such that for all $\beta \geq \beta_0$, $\mu$ is an equilibrium state for $\beta \phi$. It is not difficult to show that a potential $\phi$ freezes if and only if $P_\phi$ attains its slant asymptote for some $\beta_0 \geq 0$, and we provide a proof of the statement in Lemma \ref{freeze iff asymptote}. We note that when a potential freezes on a collection of equilibrium states, this collection is trivially a collection of zero-temperature limits for $\phi$. In particular, it must be the case that these measures are $\phi$-maximizing and entropy is constant among this collection. 

In \cite{hofbauer}, Hofbauer constructed a potential on the full shift $(\{ 0, 1\}^\Z, \sigma)$ that froze at the delta measure of the sequence of all $1$'s ($\delta_1$). Lopes extended this in \cite{lopes} and was able to construct a potential that froze on the convex hull of $\delta_1$ and $\delta_0$. In \cite{van-Enter}, van Enter and Mi\c{e}kisz asked if it was possible for a potential to freeze on a measure that is fully supported on a quasicrystal, which was explored by Bruin and Leplaideur in \cite{bruin-quasicrystal} and \cite{bruin-freezing}. Additionally, the construction in \cite{multiple-phase-transitions} allows one to find a potential on the full shift $(\{ 0, 1\}^\Z, \sigma)$ that freezes on the measure of maximal entropy for any subshift $X$. Outside of the subshift setting it was shown in \cite{freezing-Z} that when $T$ is a $\Z$-action, the entropy map $h$ is upper semi-continuous, and $(X, T)$ has finite topological entropy, then for any zero-entropy ergodic measure $\mu$, there exists a potential $\phi$ that freezes on $\mu$. 

\subsection{Our Results}

Our first result greatly extends and generalizes the affirmative answer to the question posed in \cite{van-Enter}, showing that for any dynamical system, any collection of equilibrium states where $h$ is constant can be obtained as freezing states for some potential. 


\begin{fact1}\label{face freezing} Let $(X, T)$ be any dynamical system and let $\mathcal{F}$ be a non-empty subset of $M_T(X)$. Then there there exists $\psi \in C(X)$ such that for all $\beta \geq 1$, $\mathcal{F}$ is the collection of equilibrium states for $\beta \psi$ if and only if the following two conditions hold: 
\begin{itemize}
\item For all $\mu, \nu \in \mathcal{F}$, $h(\mu) = h(\nu)$ and 
\item there exists some $\phi \in C(X)$ such that $\mathcal{F}$ is the collection of equilibrium states for $\phi$. 
\end{itemize}
\end{fact1}


In fact, the above conditions are necessary and sufficient for a collection of measures to be the freezing states for some potential. As an immediate corollary, we see that for a fixed ergodic measure, these properties are equivalent. 

\begin{cor1} Let $(X, T)$ be any dynamical system and let $\mu \in M_{erg}(X)$. Then $\mu$ is an equilibrium state for some potential if and only if $\mu$ is the freezing state for some potential. 
\end{cor1}

We then use Theorem \ref{face freezing} to prove a similar result to that of Jenkinson and Ruelle in the case where $h$ is upper semi-continuous. 

\begin{cor1}\label{usc case} Let $(X, T)$ be any dynamical system such that the entropy map is upper semi-continuous. Let $\mathcal{E} \subset M_{erg}(X)$ be a non-empty collection of ergodic measures so that $\mathcal{E}$ is closed in $M_T(X)$ and suppose that $h$ is constant on $\mathcal{E}$. Then there exists $\phi \in C(X)$ such that for all $\beta \geq 1$, $\overline{co}(\mathcal{E})$ is the collection of equilibrium states for $\beta \phi$. 
\end{cor1}

We note here that the hypotheses for Corollary \ref{usc case} are satisfied in many cases of interested, including where $(X, T)$ is a $d$-dimensional shift system over a finite alphabet, corresponding to lattice systems in statistical physics including the Ising and Potts models. 

Next, we turn to the structure of the set of freezing potentials in $C(X)$. It is well known that the set of potentials with unique equilibrium states is dense in $C(X)$ when $h$ is upper semi-continuous \cite{walters}, and we prove an analogous result:  

\begin{fact1} Let $(X, T)$ be a dynamical system such that the entropy map is upper semi-continuous. Then the set of potentials that freeze is dense in $C(X)$ with the uniform topology. 
\end{fact1}

However, in the case where $T$ is a $\Z$-action and $(X, T)$ has the specification property, even though the collection of potentials that freeze are dense, the complement is residual.

\begin{fact1} Let $(X, T)$ be a dynamical system where $T$ is a $\Z$ action and $(X, T)$ has the specification property. Then the collection of potentials that do not freeze contains a dense $G_\delta$. 
\end{fact1}

\setcounter{fact1}{0}

%
%

The structure of the paper is as follows. We begin with preliminaries in Section 2, introducing formal definitions for the relevant concepts. In Section 3 we prove the main results of this paper, and we conclude in Section 4 with a collection of observations we thought useful to include in this paper for additional context. These include that a potential inducing a non-trivial freezing state inherently induces a phase transition where $P_\phi$ is not real analytic at the freezing temperature, and showing that it is possible to construct a potential $\phi$ that exhibits a rapid phase transition, where the equilibrium states for $\beta \phi$ switch between a measure of maximal entropy to a zero-entropy measure at $\beta_0 = 1$.

\subsection*{Acknowledgements} This work was conducted during the author's PhD studies under the supervision of Prof. Ronnie Pavlov. The author would like to thank Prof. Pavlov for all of his assistance throughout the research and writing process. The author would also like to thank Prof. van Enter and Prof. Chazottes for their helpful suggestions.

\section{Preliminaries}

\subsection{Dynamical Systems} 

We fix a compact metric space $X$, and let $G$ be a topological (semi)-group acting on $X$ continuously by $\{ T_g \}_{g \in G}$. A {\bf dynamical system} is the pair $(X, T)$. We say a Borel measure $\mu$ on $X$ is {\bf$T$-invariant} if for all measurable $A \subset X$ and for all $g \in G$, we have $\mu(T_g^{-1} A ) = \mu(A)$. Let $M_T(X)$ denote the collection of all $T$-invariant Borel probability measures on $X$. For the remainder of this paper we assume that $M_T(X)$ is non-empty and some notion of measure theoretic entropy is defined, which is always the case when $G$ is an amenable group. This includes, for example, when the action is induced by $\R$, $\Z^d$. 

It is well known that, since $X$ is a compact metric space, $M_T(X)$ is weak-$*$ compact, convex, and in fact forms a Choquet simplex, where the extreme points of $M_T(X)$ are exactly the ergodic measures for the dynamical system $(X, T)$. We say an invariant measure $\mu$ is {\bf ergodic} when for any $A \subset X$, if for all $g \in G$, $\mu(T_g^{-1} A \triangle A) = 0$, then $\mu(A) \in \{ 0, 1 \}$. We denote the set of ergodic measures by $M_{erg}(X)$ (where $T$ is implicit). 

For a given invariant measure $\mu$, we denote the Kolmogorov-Sinai {\bf entropy of $\mu$} by $h(\mu)$ (see \cite{walters} for a detailed definition in the discrete case and \cite{flow-entropy} for $G = \R$). It is well known that $h: M_T(X) \rightarrow \R$ is non-negative and affine. We now define the {\bf topological entropy} of $(X, T)$ by $h(X) = \sup \{ h(\mu) : \mu \in M_T(X) \}$. 
For the remainder of this paper we assume that $h(X) < \infty$. We would like to note here that the results of this paper can be applied to any bounded, affine $h: M_T(X) \rightarrow \R$, including variations of the definition of measure theoretic entropy and generalizations of it.

We let $C(X)$ represent the collection of real-valued continuous functions on $X$ and we call any $\phi \in C(X)$ a {\bf potential}. When equipped with the supremum norm, $C(X)$ forms a Banach space and its (real) dual $C(X)^*$ corresponds with the signed Borel measures on $X$. We let $(C(X)^*, w^*)$ represent the dual of $C(X)$ endowed with the weak-$*$ topology. 

Note here that the dual of the topological vector space $(C(X)^*, w^*)$ is in direct correspondence with $C(X)$, where for each $F \in (C(X)^*, w^*)^*$, there exists some $f \in C(X)$ such that for all $\mu \in C(X)^*$, $F(\mu) = \int f d\mu$ (and of course vice versa). By a slight abuse of notation, we denote $(C(X)^*, w^*)^* $ by its representation in $C(X)$. Additionally, by a straightforward application of Hahn-Banach, for any $w^*$ continuous and affine $F: M_T(X) \rightarrow \R$, we can linearly extend $F$ to $(C(X)^*, w^*)$ and then represent $F$ by some $f \in C(X)$. This representation need not be unique, but for the purposes of this paper it is sufficient to note that for all affine, weak-$*$ continuous $F: M_T(X) \rightarrow \R$, there exists at least one $f \in C(X)$ such that for all $\mu \in M_T(X)$, $F(\mu) = \int f d\mu$. 

An essential element the techniques of this paper will be Theorem 4 from \cite{edwards}, known as Edward's Theorem. For our purposes the theorem can be stated as follows: we let $K$ be a Choquet simplex and let $-f, g : K \rightarrow (-\infty, \infty]$ be concave weak-$*$ lower semi-continuous  functions such that $f \leq g$. Then there exists a weak-$*$ continuous, affine $h: K \rightarrow \R$ such that $f \leq h \leq g$.


\subsection{Pressure and Equilibrium States} 
When viewing potentials as energy functions, we can extend the idea of entropy to that of pressure, an analogue of the concept of free energy. For a potential $\phi \in C(X)$, the {\bf pressure} of $\phi$ can be defined via the variational principle: 
$$P(\phi) = \sup \{ h(\mu) + \int \phi d\mu : \mu \in M_T(X) \}. $$

Equilibrium states are defined as probability distributions that maximize the pressure of the system. In particular, we say for a given $\mu \in M_T(X)$, {\bf $\mu$ is an equilibrium state for $\phi$} if it attains the supremum: 
$$P(\phi) =  h(\mu) + \int \phi d\mu. $$
For a fixed potential $\phi$, let $\mathcal{F}$ be the collection of equilibrium states for $\phi$. In the case where $h$ is upper semi-continuous, it can be easily shown that $\mathcal{F}$ is non-empty, compact, and convex. In particular, it is an exposed face of $M_T(X)$ and the extreme points of $\mathcal{F}$ are exactly the ergodic equilibrium states. When $h$ is not upper semi-continuous, it is not necessarily the case that $\mathcal{F}$ is non-empty or closed. $\mathcal{F}$ is still convex in this general case, and it is easy to show that the extreme points are exactly the ergodic equilibrium states, but beyond this, little else is known about $\mathcal{F}$ in general.

%

We provide here two simple examples showing that the collection of equilibrium states for the potential $\phi = 0$ need not be closed, and may be empty. First we let $\N^*$ denote the one point compactification of $\N$ (i.e. $\N^* = \N \cup \{ \infty \}$). We do not provide a proof for the statements in the examples for the sake of space. 

\begin{ex} Let $X \subset ( \N^*)^\Z$ be the countable state Markov chain such that for all $x \in X$, for all $n \in \Z$, $|x_{n+1} - x_n| \leq 1$. Under the left shift map, $(X, \sigma)$ is a dynamical system with topological entropy $\log 3$. However, for any $\sigma$-invariant measure $\mu \in M_T(X)$, $h(\mu) < \log 3$. It is therefore the case that the set of equilibrium states for the potential $\phi = 0$ is empty. 
\end{ex}

\begin{ex}\label{disjoint subshift} Let $X \subset (\N^*)^\Z$ be the disjoint union of the full shifts on $\{ 2n, 2n+1 \}$ unioned with the fixed point $\infty^\Z$. Clearly the entropy of this Markov chain is $\log 2$, and the collection of equilibrium states for $\phi = 0$ is exactly the convex hull of the measures of maximal entropy supported on each disjoint full shift. However, this collection is not closed in $M_T(X)$ since $\delta_\infty$ is in the closure yet $h(\delta_\infty) = 0$. 
\end{ex}

We note here that Example \ref{disjoint subshift} could be altered so that $X$ is topologically mixing by using a construction similar to Haydn in \cite{haydn}, but save the details for the interested reader.

\subsection{Zero Temperature Limits and Freezing Behavior}

We let $\beta > 0$ represent the inverse temperature of our system. We can examine zero-temperature limits of a potential $\phi$ by looking at equilibrium states for $\beta \phi$ as we let $\beta \nearrow \infty$. We say $\mu_\infty$ is a {\bf zero-temperature limit} of $\phi$ if there exists a collection of equilibrium states $\mu_\beta$ for $\beta \phi$ such that $\mu_\beta \rightarrow \mu_\infty$ as $\beta \nearrow \infty$. It is easy to show that the collection of zero-temperature limits are entropy maximizing among the $\phi$-maximizing invariant measures. We say that a potential {\bf $\phi$ freezes at $\mu$} if there exists some $\beta_0 > 0$ such that for all $\beta \geq \beta_0$, $\mu$ is an equilibrium state for $\beta\phi $.

\section{Proof of Results}

\subsection{Equilibrium States are Freezing States}

We begin by identifying necessary conditions required for a potential $\phi$ to freeze at a set of invariant measures $\mathcal{F}$.  

\begin{lemma}\label{beta eq states} Let $(X, T)$ be any dynamical system. Let $\phi \in C(X)$ and let $\mathcal{F}$ be a non-empty subset of $M_T(X)$. Then for all $\beta \geq 1$, $\mathcal{F}$ is the collection of equilibrium states for $\beta \phi$ if and only if the following conditions hold: 
\begin{enumerate}
\item $\mathcal{F}$ is the collection of equilibrium states for $\phi$, 
\item for all $\mu \in \mathcal{F}$, $\int \phi d\mu = \sup_{\nu \in M_T(X)} \int \phi d\nu$. 
\end{enumerate}
\end{lemma}

\begin{proof} Let $c \in \R$ and notice that for any $\beta > 0$, $\mu$ is an equilibrium state for $\beta \phi$ if and only if it is an equilibrium state for $\beta \left( \phi + c \right)$. We may therefore assume without loss of generality that $\sup_{\nu \in M_T(X)} \int \phi d\nu = 0$. 

We now suppose (1) and (2) hold. Since by assumption for any $\mu, \nu \in \mathcal{F}$, $\int \phi d\mu = \int \phi d\nu$ and $\mu$ and $\nu$ are both equilibrium states for $\phi$, it immediately follows that $h$ is constant on $\mathcal{F}$. 

We now let $\beta \geq 1$, $\nu \in M_T(X) \backslash \mathcal{F}$, and $\mu \in \mathcal{F}$. Notice here by assumption 
$$h(\nu) + \int \beta \phi d\nu = h(\nu) + \int \phi d\nu + (\beta - 1) \int \phi d\nu$$
$$\leq  h(\nu) + \int \phi d\nu < h(\mu) + \int \phi d\mu = h(\mu) + \int \beta \phi d\mu. $$
It immediately follows that the collection of equilibrium states for $\beta \phi$ is contained in $\mathcal{F}$. 
However, for any $\mu, \nu \in \mathcal{F}$ we know that
$$h(\nu) + \int \beta \phi d\nu = h(\nu) = h(\mu) = h(\mu) + \int \beta \phi d\mu. $$
We can now conclude that $\mathcal{F}$ is exactly the collection of equilibrium states for $\beta \phi$. 

$\;$

We now suppose that for all $\beta \geq 1$, $\mathcal{F}$ is the collection of equilibrium states for $\beta \phi$. Trivially by taking $\beta = 1$, condition (1) holds. We now suppose that (2) does not hold. In particular, we have some $\mu \in \mathcal{F}$ such that $\int \phi d\mu < 0 = \sup_{\nu \in M_T(X)} \int \phi d\nu$. Fix this $\mu$ and let $\beta > 1$ be sufficiently large such that $\beta \int \phi d\mu < - h(X)$. We therefore know that 
$$h(\mu) + \int \beta \phi d\mu < 0 \leq h(\nu) +  \int \beta \phi d\nu $$
for any $\nu$ satisfying $\int \phi d\nu = 0$. We therefore know $\mu$ is not an equilibrium state for $\beta \phi$, arriving at a contradiction.


\end{proof}

%
%
%
%

We may now prove our first result: 


\begin{fact1}\label{face freezing} Let $(X, T)$ be any dynamical system and let $\mathcal{F}$ be a non-empty subset of $M_T(X)$. Then there there exists $\psi \in C(X)$ such that for all $\beta \geq 1$, $\mathcal{F}$ is the collection of equilibrium states for $\beta \psi$ if and only if the following two conditions hold: 
\begin{itemize}
\item For all $\mu, \nu \in \mathcal{F}$, $h(\mu) = h(\nu)$ and 
\item there exists some $\phi \in C(X)$ such that $\mathcal{F}$ is the collection of equilibrium states for $\phi$. 
\end{itemize}
\end{fact1}

\begin{proof} First suppose that $\mathcal{F}$ is the collection of equilibrium states for $\beta \psi$ for all $\beta \geq 1$. Note trivially that $\mathcal{F}$ is the collection of equilibrium states for $\phi = 2 \psi$. Additionally, we note that every $\mu \in \mathcal{F}$ is a ground state for $\psi$, and therefore it must be the case that $h(\mu) = h_\infty(\psi)$. It immediately follows that $h$ must be constant on $\mathcal{F}$.

We suppose that $h$ is constant on $\mathcal{F}$ and there exists some $\phi \in C(X)$ such that $\mathcal{F}$ is the collection of equilibrium states for $\phi$. Note here that for all $\nu, \mu \in \mathcal{F}$, 
$$h(\nu) + \int \phi d\nu = h(\mu) + \int \phi d\mu. $$
Since by assumption $h(\nu) = h(\mu)$, we may assume without loss of generality by adding a scalar that for all $\mu \in \mathcal{F}$, $\int \phi d\mu = 0$. We now define the following functions $f, g : M_T(X) \rightarrow \R$ such that 
\begin{itemize}
\item For all $\nu \in M_T(X)$, $f(\nu) = | \int \phi d\nu |$ and 
\item $g|_{\overline{\mathcal{F}}} = 0$ and $g|_{M_T(X) \backslash \overline{\mathcal{F}} } = ||\phi||$. 
\end{itemize}
It is easy to see that $f$ is convex and weak-$*$ continuous. Additionally, since $\mathcal{F}$ is convex, $\overline{\mathcal{F}}$ is a closed and convex set and therefore $g$ is concave and weak-$*$ lower semi-continuous. By construction we know $f \leq g$ and since $M_T(X)$ is a Choquet simplex, we can apply Theorem 4 from \cite{edwards} to find a weak-$*$ continuous affine $F: M_T(X) \rightarrow \R$ such that $f \leq F \leq g$. In particular, $F$ can be identified with some $F \in C(X)$.

We let $\psi = \phi - F$. Let $\mu \in \mathcal{F}$ and examine: 
$$\int \psi d\mu = \int \phi d\mu - \int F  d\mu = 0. $$
Additionally, for all $\nu \in M_T(X) \backslash \mathcal{F}$, we have 
$$\int \psi d\nu = \int \phi  d\nu - \int F  d\nu \leq \int \phi d\nu - |\int \phi d\nu| \leq 0. $$

Thus condition (2) of Lemma \ref{beta eq states} holds and it remains to show that $\mathcal{F}$ is the collection of equilibrium states for $\psi$. We now let $\mu \in \mathcal{F}$ and $\nu \in M_T(X) \backslash \mathcal{F}$ and examine 
$$h(\nu) + \int \psi d\nu = h(\nu) + \int \phi d\nu - \int  Fd\nu  u $$
$$< h(\mu) + \int \phi d\mu - \int F d\nu  \leq h(\mu) + \int \phi d\mu- \int F d\mu = h(\mu) + \int \psi d\mu.  $$
It therefore follows that $\mathcal{F}$ is the collection of equilibrium states for $\psi$. We can now apply Lemma \ref{beta eq states} to see that for all $\beta \geq 1$,  $\mathcal{F}$ is the collection of equilibrium states for $\beta \psi$. 

\end{proof}

As immediate corollaries, we obtain results in the case where $\mathcal{F}$ is a single ergodic measure, as well as a result in the case where $h$ is upper semi-continuous.  

\begin{cor1} Let $\mu \in M_{erg}(X)$. Then $\mu$ is an equilibrium state for some potential if and only if $\mu$ is the freezing state for some potential. 
\end{cor1}


\begin{cor1}\label{usc case} Let $(X, T)$ be any dynamical system such that the entropy map is upper semi-continuous. Let $\mathcal{E} \subset M_{erg}(X)$ be a non-empty collection of ergodic measures so that $\mathcal{E}$ is closed in $M_T(X)$ and suppose that $h$ is constant on $\mathcal{E}$. Then there exists $\phi \in C(X)$ such that for all $\beta \geq 1$, $\overline{co}(\mathcal{E})$ is the collection of equilibrium states for $\beta \phi$. 
\end{cor1}

\begin{proof} This follows immediately by direct application of Theorem 5 from \cite{jenkinson} and Theorem \ref{face freezing} above. 
\end{proof}

%
%

\subsection{Density in $C(X)$}

We will now show that the collection of potentials that freeze at a singleton is dense in $C(X)$ with the uniform topology. 

\begin{fact1} Let $(X, T)$ be a dynamical system such that the entropy map is upper semi-continuous. Then the set of potentials that freeze is dense in $C(X)$ with the uniform topology. 
\end{fact1}

\begin{proof} Let $\phi \in C(X)$ and let $\mu_\infty$ be any ergodic measure that is $\phi$-maximizing. Since $\mu_\infty$ is ergodic and $h$ is upper semi-continuous, by Corollary \ref{usc case} there must be a potential $\psi$ that freezes on $\{ \mu_\infty \}$. By adding a constant we may assume without loss of generality that $\int \psi d\mu_\infty = 0$. We claim that for any $\epsilon > 0$, $\phi + \epsilon \psi$ freezes on $\mu_\infty$. 

First we fix $\beta_0 > 0$ such that $\psi$ freezes on $\mu_\infty$ at $\beta_0$. Let $\nu \neq \mu_\infty$ and $\beta > \beta_0$ and examine 
$$h(\nu) + \int  \frac{\beta}{\epsilon}  \left( \phi + \epsilon \psi \right) d\nu = h(\nu) + \int \beta \psi d\nu +  \int \frac{\beta}{\epsilon}  \phi d\nu $$
$$< h(\mu_\infty) + \int \beta \psi d\mu_\infty + \int \frac{\beta}{\epsilon}  \phi d\nu \leq h(\mu_\infty) + \int \beta \psi d\mu_\infty + \int \frac{\beta}{\epsilon}  \phi d\mu_\infty . $$
It immediately follows that $\phi + \epsilon \psi$ freezes on $ \{ \mu_\infty \}$ (at $\beta_0 / \epsilon$), and we can conclude the desired results. 
\end{proof}


\subsection{Non-Freezing Potentials are Generic under Specification} 

Finally, we show that although the collection of potentials that freeze are dense whenever $h$ is upper semi-continuous, under specification the complement is residual. 

\begin{fact1}\label{generic} Let $(X, T)$ be a dynamical system where $T$ is a $\Z$ action and $(X, T)$ has the specification property. Then the collection of potentials that do not freeze contains a dense $G_\delta$. 
\end{fact1}

\begin{proof} For a given $\phi \in C(X)$, let $Max(\phi) = \sup \{ \int \phi d\mu : \mu \in M_T(X) \}$. Note that $Max ( \cdot ) : C(X) \rightarrow \R$ is continuous (in fact it is Lipschitz). It therefore follows for all $\beta > 0$, 
$$A_\beta = \{ \phi \in C(X) : P_{top}(\beta \phi) > Max(\beta \phi) \}$$
is open in $C(X)$. Additionally, by continuity of $P_{top}$ and $Max(\cdot)$, we know 
$$A = \bigcap_{\beta > 0} A_\beta = \bigcap_{\beta \in \Q^+} A_\beta. $$
Thus we know 
$$A = \{ \phi \in C(X) : \forall \beta > 0, P_{top}(\beta \phi) > Max( \beta \phi ) \} $$
is a $G_\delta$ set in $C(X)$. 

We now note that since $(X, T)$ is a dynamical system with specification, a result by Bowen in \cite{bowen} implies that the collection of H\o lder potentials is contained in $A$ and it can be shown that this is dense in $C(X)$. 

For a potential $\phi$, let $h_\infty(\phi)$ denote the residual entropy of $\phi$. In particular, it is the maximal entropy of $\phi$-maximizing measures and is obtained by any cluster point of any sequence $(\mu_\beta)$ where $\mu_\beta$ is an equilibrium state for $\beta \phi$. We now define: 
$$B = \{ \phi \in C(X) : h_\infty(\phi) = 0 \}. $$
Again since $(X, T)$ has specification, we can apply Corollary 1.3 of \cite{morris-residual-entropy} to see that $B$ contains a dense $G_\delta$. 

It follows immediately that for any $\phi \in A \cap B$, $P_\phi$ does not obtain its slant asymptote and by Lemma \ref{freeze iff asymptote} it follows that $\phi$ does not freeze and we can conclude our desired result. 
\end{proof}

\section{Additional Observations regarding Phase Transitions}

We conclude this paper by proving a few additional observations regarding phase transitions when a potential $\phi$ freezes at some $\beta_0 > 0$. 

\subsection{Freezing and Analyticity of the Pressure Function} The first lemma is known in the literature (for example \cite{multiple-phase-transitions}), but we provide a proof here. 

\begin{lemma}\label{P derivative} Fix $\phi \in C(X)$ and suppose $P_\phi$ is differentiable at $\beta_0$ and that there exists at least one equilibrium state for $\beta_0 \phi$. Then the collection of all equilibrium states for $\beta_0 \phi$ satisfy the following: 
$$\int \phi d\mu = \frac{\partial P_\phi}{\partial \beta}(\beta_0). $$
\end{lemma}

\begin{proof} We suppose that $P_\phi$ is differentiable at $\beta_0$ and we let $\mu$ be an equilibrium state for $\beta_0 \phi$. We now let $\beta > \beta_0$ and note that 
$$P_\phi(\beta) \geq h(\mu) + \int  \beta \phi d\mu. $$
We therefore have 
$$\lim_{\beta \searrow \beta_0} \frac{P_\phi(\beta) - P_\phi(\beta_0)}{\beta - \beta_0} \geq \int \phi d\mu. $$
\end{proof}

%
%
%
%
%
%
%
%

\begin{lemma}\label{freeze iff asymptote}  $\phi$ freezes at $\beta_0$ if and only if $P_\phi$ attains its slant asymptote at $\beta_0$ and there exists an equilibrium state for $\beta \phi$ for some $\beta > \beta_0$. 
\end{lemma}

\begin{proof} First suppose $\phi$ freezes at $\beta_0$ and let $\mu$ be an equilibrium state for $\beta \phi$ for all $\beta \geq \beta_0$. It immediately follows that $\mu$ is a zero-temperature limit for $\phi$ and additionally for all $\beta \geq \beta_0$, 
$$P_\phi(\beta) = h(\mu) + \beta \int \phi d\mu. $$

We now let $Max(\phi) = \sup_{\mu \in M_T(X)} \int \phi d\mu$ and $h_\infty(\phi) = \sup \{ h(\mu) : \int \phi d\mu = Max(\phi) \}$. As mentioned in the introduction, the slant asymptote for $P_\phi$ is exactly $f(\beta) = \beta \cdot Max(\phi) + h_\infty (\phi)$. Suppose now there exists some $\beta_0 > 0$ such that $P_\phi$ attains this asymptote. By a straight forward convexity argument, it must be the case that for all $\beta \geq \beta_0$, $P_\phi(\beta) = h_\infty(\phi) + \beta Max(\phi)$, and it is therefore differentiable on $(\beta_0, \infty)$ with derivative $Max(\phi)$. 

Let $\mu$ be an equilibrium state for $\beta \phi$ for some $\beta > \beta_0$, which exists by assumption. Since $\beta > \beta_0$, we know $P_\phi$ is differentiable at $\beta$ with derivative $Max(\phi)$, and so by Lemma \ref{P derivative} we know $\int \phi d\mu = Max(\phi)$. It follows that for all $\beta \geq \beta_0$, $P_\phi(\beta) = h(\mu) + \int \beta \phi d\mu$ and thus $\phi$ freezes at $\beta_0$. 
\end{proof}

Finally we show that if $\beta_0 > 0$ is the smallest inverse temperature at which $\phi$ freezes, the system exhibits a phase transition, where $P_\phi$ is not analytic at $\beta_0$. 

\begin{obs} Let $\phi$ freeze at $\beta_0 > 0$ and for all $\alpha \in [0, \beta_0)$, $\phi$ does not freeze at $\alpha$. Then $\phi$ exhibits a phase transition at $\beta_0$. 
\end{obs}

\begin{proof} Let $\beta \cdot Max(\phi)+h_\infty(\phi)$ be the slant asymptote for $P_\phi$ and notice for all $\beta \geq \beta_0$, $P_\phi(\beta) = \beta \cdot Max(\phi)+h_\infty(\phi)$. Suppose for a contradiction that $P_\phi$ is analytic at $\beta_0$. Thus we have some $\epsilon > 0$ and a sequence of real numbers $(a_n)$ such that for all $ \beta \in (\beta_0 - \epsilon, \beta_0 + \epsilon)$,  
$$P_\phi(\beta) = \sum_{n=0}^\infty a_n( \beta - \beta_0)^n. $$
Since for all $\beta \in (\beta_0, \beta_0+\epsilon)$, we know $P_\phi(\beta) = \beta \cdot Max(\phi)+h_\infty(\phi)$, we know the sequence $(a_n)$ must be exactly $a_0 = h_\infty(\phi)$, $a_1 = Max(\phi)$ and $a_n = 0$ for all $n \geq 2$. It immediately follows that for $\beta \in (\beta_0 - \epsilon, \beta_0]$, $P_\phi(\beta) = \beta \cdot Max(\phi)+h_\infty(\phi)$. This contradicts the minimality of $\beta_0$, and thus $P_\phi$ is not analytic at $\beta_0$. 
\end{proof}


\subsection{Rapid Switching Behavior}

We conclude by showing that we can obtain a rapid phase transition from entropy-dominated, to energy-dominated regimes at a specific inverse temperature. In particular, the equilibrium states shift from having maximal entropy to minimal entropy at a specific $\beta_0$. We begin with an easy to prove observation.

\begin{lemma}\label{freezing subset} Let $\phi \in C(X)$ and let $\mu \in M_{erg}(X)$ be an equilibrium state for $\phi$ such that $h(\mu) = 0$. Then for all $\beta \geq 1$, $\mu$ is an equilibrium state for $\beta \phi$. 
\end{lemma}

\begin{proof} Since $h(\mu) = 0$, it is easy to see that $\mu$ must be $\phi$-maximizing. We now let $\nu \in M_T(X)$ and $\beta \geq 1$ and notice 
$$ h(\mu) + \int \beta \phi d\mu \geq h(\nu) + \int \phi d\nu + (1- \beta) \int \phi d\mu \geq h(\nu) + \int \phi d\nu + (1-\beta) \int \phi d\nu = h(\nu) + \int \phi d\nu. $$
The desired result immediately follows. 
\end{proof}

We now use this fact to show that there exists a potential that rapidly switches between the entropy dominated to energy dominated regime. 

\begin{obs}\label{switching} Let $(X, T)$ by a dynamical system such that $h$ is upper semi-continuous. Let $\mu$ be any ergodic measure of maximal entropy and $\nu$ be any ergodic measure with $0$ entropy. Then, there exists a potential $\phi$ such that for all $\beta > 1$, $\nu$ is the unique equilibrium state for $\beta \phi$ and for all $\beta \in (0, 1)$, $\mu$ is the unique equilibrium state for $\beta \phi$. 
\end{obs}

\begin{proof} Let $\mu, \nu \in M_T(X)$ as in the statement and let $\phi \in C(X)$ such that $\mu$ and $\nu$ are the only ergodic equilibrium states for $\phi$. This exists by a direct application of Corollary 3.17 of \cite{ruelle} or Theorem 5 of \cite{jenkinson}. By adding a constant, we may also assume without loss of generality that $P(\phi) = 0$. In particular, we know 
$$\int \phi d\mu = - h(\mu). $$
We now let $\beta \in (0, 1)$ and notice
$$h(\mu) + \int \beta \phi d\mu = h(\mu) - \beta h(\mu) > 0 = h(\nu) + \int \phi d\nu.  $$
Thus $\nu$ is not an equilibrium state for $\beta \phi$ for $\beta < 1$. We now let $\gamma \in M_{erg}(X)$ such that $\gamma \neq \mu, \nu$. Since $\gamma$ is not an equilibrium state for $\phi$, we know  
$$h(\gamma) + \int \phi d\gamma < h(\mu) + \int \phi d\mu = 0. $$
It follows that 
$$\int \phi d\gamma <  - h(\gamma). $$
And thus, 
$$h(\gamma) + \int \beta \phi d\gamma < (1-\beta) h(\gamma) \leq (1-\beta) h(\mu) = h(\mu) + \int \beta \phi d\mu. $$
We can now conclude that $\mu$ is the unique equilibrium state for $\beta \phi$ for all $\beta \in (0, 1)$. 

$\;$

We now let $\beta > 1$ and note by Observation \ref{freezing subset}, for all $\beta \geq 1$, $\nu$ is an equilibrium state for $\beta \phi$. We will now show that $\nu$ is the only equilibrium state by first observing that for any $\beta > 1$, we have 
$$ h(\mu) + \int \beta \phi d\mu = (1-\beta) h(\mu) < 0 = h(\nu) + \int \beta \phi d\nu. $$
We now let $\gamma \in M_{erg}(X)$ be any other ergodic measure and note since $\gamma$ is not an equilibrium state for $\phi$, we know as above, 
$$\int \phi d\gamma <  - h(\gamma). $$
We therefore have for all $\beta > 1$, 
$$h(\gamma) + \int \beta \phi d\gamma < (1-\beta) h(\gamma) \leq 0 = h(\nu) + \int \beta \phi d\nu. $$
We can now conclude the desired result. 
\end{proof}

\subsection*{Data Availability} We do not analyze or generate any datasets, because our work proceeds within a theoretical and mathematical approach.

\subsection*{Disclosures} The author has no conflicts of interest to declare that are relevant to the content of this article. No funding was received to assist with the preparation of this manuscript.

\bibliographystyle{plain}
\bibliography{mybib}

\end{document}